\documentclass[12pt,twoside,final,a4paper]{amsart}
\usepackage[margin=1.5cm,includefoot,footskip=30pt]{geometry}
\usepackage{amsaddr}
\usepackage{theoremref}
\usepackage{amssymb}
\usepackage{lipsum}
\usepackage[all,cmtip]{xy}
\usepackage{cite}
\usepackage[utf8]{inputenc}
\usepackage{upgreek}
\usepackage{fix-cm}
\usepackage{graphicx}
\usepackage[T1]{fontenc}
\usepackage{times}
\usepackage{array}
\usepackage{epstopdf}
\usepackage{booktabs}
\usepackage[toc,page]{appendix}
\usepackage[section]{placeins}
\usepackage{longtable}

\newcommand{\diag}{\mathrm{diag}}

\newcommand{\GL}{\mathrm{GL}}
\newcommand{\GO}{\mathrm{GO}}
\newcommand{\GSp}{\mathrm{GSp}}
\def \SL {\mathrm{SL}}

\long\def\symbolfootnote[#1]#2{\begingroup%
\def\thefootnote{\fnsymbol{footnote}}\footnote[#1]{#2}\endgroup}

\newcommand{\tr}{\ensuremath{{}^T\!\!}}

\newcommand{\C}{\mathfrak C}

\newcommand{\E}{\mathcal E}
\newcommand{\F}{\mathcal F}

\newcommand{\tensor}{\otimes}
\newtheorem{theorem}{Theorem}[section]
\newtheorem{lemma}[theorem]{Lemma}
\newtheorem{corollary}[theorem]{Corollary}
\newtheorem{proposition}[theorem]{Proposition}
\newtheorem{definition}{Definition}[section]

\newtheorem*{theorema}{Theorem A}
\newtheorem*{corollarya1}{Corollary A1}
\newtheorem*{corollarya2}{Corollary A2}
\newtheorem*{corollarya3}{Corollary A3}

\theoremstyle{remark}
\newtheorem*{remark}{Remark}
\numberwithin{equation}{section}

\newcommand{\ignore}[1]{}

\newcommand{\mynote}[1]{}

\usepackage{mathptmx}
\makeatletter
\def \namedlabel#1#2{\begingroup\def \@currentlabel{#2}
\label{#1}\endgroup}

\makeatletter
\def\Ddots{\mathinner{\mkern1mu\raise\p@
\vbox{\kern7\p@\hbox{.}}\mkern2mu
\raise4\p@\hbox{.}\mkern2mu\raise7\p@\hbox{.}\mkern1mu}}
\makeatother

\begin{document}
\title{Algorithms in Linear Algebraic Groups}
\author[Bhunia, Mahalanobis, Shinde and Singh]{Sushil Bhunia, Ayan Mahalanobis, Pralhad Shinde, and Anupam Singh}
\AtEndDocument{\bigskip{\footnotesize%
\address{Address: IISER Mohali, Knowledge city, Sector 81, SAS Nagar, P.O. Manauli, Punjab 140306, INDIA.}\par 
\email{Email address: sushilbhunia@gmail.com}\par 
\address{Address: IISER Pune, Dr.~Homi Bhabha Road, Pashan, Pune 411008, INDIA.}\par 
\email{Email address: ayan.mahalanobis@gmail.com}\par 
\email{Email address: pralhad.shinde96@gmail.com}\par 
\email{Email address: anupamk18@gmail.com}\par 
}}
%\subjclass[2010]{11E57, 15A21}
%\today
\date{}
%%%%%%%%%%%%%%%%%%%%%%%%%%%%
\begin{abstract} This paper presents some algorithms in linear algebraic groups. These algorithms solve the word problem and compute the spinor norm for orthogonal groups. This gives us an algorithmic definition of the spinor norm. We compute the double coset decomposition with respect to a Siegel maximal parabolic subgroup, which is important in computing infinite-dimensional representations for some algebraic groups.
\end{abstract}
\maketitle
\section{Introduction}
Spinor norm was first defined by Dieudonn\'{e} and Kneser using Clifford algebras. Wall~\cite{wa} defined the spinor norm using bilinear forms. These days, to compute the spinor norm, one uses the definition of Wall. In this paper, we develop a new definition of the spinor norm for split and twisted orthogonal groups. Our definition of the spinor norm is rich in the sense, that it is algorithmic in nature. Now one can compute spinor norm using a Gaussian elimination algorithm that we develop in this paper. 
  This paper can be seen as an extension of our earlier work in the book chapter~\cite{bmss}, where we described Gaussian elimination algorithms for orthogonal and symplectic groups in the context of public key cryptography.

 In computational group theory, one always looks for algorithms to solve the word problem. For a group $G$ defined by a set of generators $\langle X\rangle=G$, the problem is to write $g\in G$ as a word in $X$: we say that this is the \textbf{word problem} for $G$ (for details, see~\cite[Section 1.4]{ob3}). Brooksbank~\cite{brooksbank} and Costi~\cite{costi}  developed algorithms similar to ours for classical groups over finite fields. It is worth noting that, Chernousov et.~al.~\cite{ceg2000} also used Steinberg presentation for Gauss decomposition for Chevalley groups over arbitrary fields. We refer the reader to a book by Carter~\cite[Theorem 12.1.1]{ca} for Steinberg presentation. 
We prove the following: 
\begin{theorema}\namedlabel{maintheorem}{Theorem A}
Let $G$ be a reductive linear algebraic group defined over an algebraically closed field $k$ of $\mathrm{char}\; \neq 2$ which has the Steinberg presentation. 
Then every element of $G$ can be written as a word in Steinberg generators and a diagonal matrix. 
The diagonal matrix is:  $\diag(\alpha,1,\ldots,1,\lambda,\mu(g),\ldots, \mu(g),\mu(g)\lambda^{-1})$, where $\alpha, \lambda, \mu(g)\in k^{\times}$, in its natural presentation. Furthermore, we prove that the length of the word is bounded by $\mathcal{O}(l^3)$, where $l$ is the rank of the group $G$.
\end{theorema}
\noindent
 We prove this theorem in Section~\ref{proofmaintheorem}. The proof is algorithmic in nature. The method we develop is Gaussian elimination algorithm to solve this problem. Steinberg generators are also called elementary matrices (for details, see Sections~\ref{elementarymate},~\ref{elementarymatsp},~\ref{elementarymato} and~\ref{elementarymattwist}).
A novelty of our algorithm is that we do not need to assume that the Steinberg generators generate the group under consideration. Thus our algorithm independently proves the fact that Chevalley groups are generated by elementary matrices.  Also, this paper can be seen as developing a Gaussian elimination algorithm in reductive algebraic groups. This is to our knowledge the first attempt to develop a Gaussian elimination algorithm for reductive algebraic groups over an  algebraically closed field.

Now we move on to discuss two applications of our algorithm. One is spinor norm and the other is double coset decomposition.
 Murray and Roney-Dougal~\cite{mr} studied computing spinor norm earlier. From our algorithm, one can compute the spinor norm easily (for details see Section~\ref{spinornorm}).~\ref{maintheorem} has the following surprising corollary: 
\begin{corollarya1}\namedlabel{corollary}{Corollary A1}
Let $k$ be a field of ${\rm char} \neq 2$. In the split orthogonal group $\mathrm{O}(n, k)$, the image of $\lambda$ in $k^\times/k^{\times 2}$ is the spinor norm. 
\end{corollarya1}
\noindent
We prove this corollary in Section~\ref{proofcorollary}.
 Since the commutator subgroup of the orthogonal group is the kernel of the spinor norm restricted to the special orthogonal group, the above corollary also gives a membership test for the commutator subgroup in the orthogonal group. In other words, an
element $g$ in the special orthogonal group belongs to the commutator subgroup if and only if the $\lambda$ produced in the Gaussian elimination algorithm is a square in the field.

Furthermore, the spinor norm can also be computed using our algorithm for the twisted orthogonal group. For terminologies of the next result, we refer to Definition~\ref{to-} and Section~\ref{elementarymattwist}.
\begin{corollarya2}\namedlabel{corollary11}{Corollary A2}
Let $g\in \mathrm{O}^{-}(2l,q)$, then the spinor norm $\Theta(g)$ of $g$ is the following:
\begin{equation*}\label{twistspinornorm}
\Theta(g)=\left\{\begin{array}{ll}
\lambda(1-t)\mathbb{F}_q^{\times 2} & \text{when}\; \mathrm{det}(g)=1,\\
2\epsilon\lambda(1-t)\mathbb{F}_q^{\times 2} & \text{when}\; \mathrm{det}(g)=-1.
\end{array}\right.
\end{equation*}
Here $t\in \mathbb{F}_q$.
\end{corollarya2}
We prove this corollary in Section~\ref{proofcorollary11}.
So, we have an efficient algorithm to compute the spinor norm.

Suppose we want to construct infinite-dimensional representations of reductive linear algebraic groups. One way to construct such representations is parabolic induction. Let $P$ be a parabolic subgroup of $G$ with Levi decomposition $P=MN$, where $M$ is the maximal reductive subgroup of $P$ and $N$ is the unipotent radical of $P$. Then a representation of the Levi subgroup $M$ can be inflated to $P$ which acts trivially on $N$. Then we use the parabolic induction to get a representation of $G$ from the representation of $P$ actually from $M$. For instance, one uses the Siegel maximal parabolic subgroups to construct infinite-dimensional representations. Since the same Levi subgroup can lie in two non-conjugate parabolic subgroups, one uses double coset decomposition of $P$ to remedy the situation. Therefore the Levi decomposition as above does not depend on the choice of a parabolic containing $M$. 
Our algorithm can be used to compute the \textbf{double coset decomposition} corresponding to the Siegel maximal parabolic subgroup (for details, see \cite{ca2}). We have the following: 
\begin{corollarya3}\namedlabel{thmbruhatdecomposition}{Corollary A3}
	Let $P$ be the Siegel maximal parabolic subgroup in $G$, where $G$ is either $\mathrm{O}(n,k)$ or
	$\mathrm{Sp}(n,k)$.  Let $g\in G$. Then there is an efficient algorithm
	to determine $\omega$ such that $g\in P\omega P$. Furthermore, $\widehat
	W$ the set of all $\omega$ is a finite set of $l+1$ elements, where $n=2l$ or $2l+1$.
\end{corollarya3}
We prove this corollary in Section~\ref{proofbruhat}. We hope this will shed some light in the infinite-dimensional representations of linear algebraic groups.
%%%%%%%%%%%%%%%%%%%%%%%%%%%%%%%%%%%%%%%%%%%%%%%%%%%%%%%%%%%%%%%%%%%
\section{Preliminaries}\label{basics}
In this section, we fix some notations and terminologies for this paper. We denote the transpose of a matrix $X$ by $\tr{X}$.
\subsection{Algebraic groups}
Algebraic groups have a distinguished history. Their origin can be traced back to the work of Cartan and Killing but we do not discuss the history of the subject here as it is quite complex. Here we just  mention Chevalley who have made pioneering contributions to this field, see for example~\cite{chevalleytohoku,chevalley}. In this paper, we develop some algorithms for reductive linear algebraic groups.
There are several excellent references on this topic. Here we follow Humphreys~\cite{humphreys}.
We fix a perfect field $k$ of $\mathrm{char} \neq 2 $ for this section, 
and $\bar{k}$ denotes the algebraic closure of $k$.
An \emph{algebraic group} $G$ defined over $\bar{k}$ is a 
group as well as an affine variety over $\bar{k}$ such that the maps 
$\mu \colon G\times G \rightarrow G$, and $i \colon G \rightarrow G$ 
given by $\mu(g_1,g_2)=g_1g_2$, and $i(g)=g^{-1}$ are morphisms of 
varieties. An \emph{algebraic group} $G$ is defined over $k$, 
if the polynomials defining the underlying affine variety $G$ 
are defined over $k$, with the maps $\mu$ and $i$ defined over 
$k$, and the identity element $e$ is a $k$-rational point of $G$. 
We denote the $k$-rational points of $G$ by $G(k)$. 
Any algebraic group
$G$ is a closed subgroup of $\GL(n,k)$ for some $n$. 
Hence algebraic groups are called \emph{linear algebraic groups}.

 The \emph{radical} of an algebraic group $G$ over $k$ is defined 
to be the largest closed, connected, solvable, normal subgroup of $G$, denoted by $R(G)$.
We call $G$ to be a \emph{semisimple} algebraic group if $R(G)=\{e\}$. 
The \emph{unipotent radical} of $G$ is defined to be the largest, closed, connected, unipotent, normal subgroup of $G$, denoted by $R_u(G)$.
We call a connected group $G$ to be \emph{reductive} if $R_u(G)=\{e\}$. 
For example, the group $\GL(n,k)$ is a reductive group, whereas $\SL(n,k)$ is 
a semisimple group. A semisimple algebraic group is always a reductive group.
In the next section, we see more examples of algebraic groups, namely, classical groups.
\subsection{Similitude groups}
In this section, we follow Grove~\cite{gr} and Knus et al.~\cite{KMRT} and 
define two important classes of groups which preserve a certain bilinear form.
Let $V$ be an $n$-dimensional vector space over $k$, where  $n=2l$ or $n=2l+1$ and $l\geq 1$.  Let $\beta\colon V\times V \rightarrow k$ be
a bilinear form. By fixing a basis of $V$ we can associate a matrix to $\beta$. With abuse of notation, we denote the matrix of the bilinear form by $\beta$ itself. Thus $\beta(x,y)=\tr x\beta y$, where $x,y$ are column vectors. We work with the non-degenerate bilinear forms, i.e., $\det\beta\neq 0$. A symmetric (resp. skew-symmetric) bilinear form $\beta$ satisfies $\beta=\tr\beta$ (resp. $\beta=-\tr\beta$). By fixing a basis for $V$, we identify $\mathrm{GL}(V)$ with $\mathrm{GL}(n, k)$ and treat symplectic and orthogonal similitude groups as subgroups of the general linear group $\mathrm{GL}(n, k)$. 
\subsubsection{Symplectic similitude groups}
Up to equivalence, there is a unique non-degenerate skew-symmetric bilinear form over a field $k$~\cite[Corollary 2.12]{gr}. Moreover, a non-degenerate skew-symmetric  bilinear form exists only in even dimension. Fix a basis of $V$ as $\{e_1,\ldots, e_l, e_{-1},\ldots, e_{-l}\}$ so that the matrix $\beta$ is: 
\begin{equation}\label{beta1}
\beta=\begin{pmatrix}0&I_l\\ -I_l&0\end{pmatrix}.
\end{equation}
\begin{definition}\label{defsymplecticgroup}
The {\normalfont symplectic group} is defined for $n=2l$ as 
$$\mathrm{Sp}(n,k):=\{ g \in \GL(n,k) \mid \tr g\beta g=\beta \}, 
\text{ where } \beta=\begin{pmatrix}
0 & I_l \\ -I_l & 0
\end{pmatrix}.$$	
\end{definition}

\begin{definition}
	The {\normalfont symplectic similitude group} with respect to $\beta$ 
	(as in Equation \eqref{beta1}), is defined by
	\[\GSp(n,k)=\{g\in \GL(n,k) \mid \tr g \beta g=\mu(g)\beta,\, \text{ for some } \mu(g) \in k^{\times}\},\] 
	where $\mu : \GSp(n,k) \rightarrow k^{\times}$  
	is a group homomorphism with $\mathrm{ker}\;\mu=\mathrm{Sp}(n,k)$ and the factor $\mu(g)$ is called the \textit{multiplier} of $g$.
\end{definition}
\subsubsection{Orthogonal similitude groups}  
We work with the following non-degenerate symmetric bilinear forms: 
Fix a basis $\{e_0,e_1,\ldots,e_l, e_{-1},\ldots, e_{-l}\}$ for odd dimension and $\{e_1,\ldots, e_l, e_{-1},\ldots, e_{-l}\}$ for even dimension so that the matrix $\beta$ is: 
\begin{equation}\label{beta2}
\beta=\left\{\begin{array}{ll}
\begin{pmatrix}0&I_l\\ I_l&0\end{pmatrix} & \text{when}\; n=2l,\\
\begin{pmatrix}2&0&0\\ 0&0&I_l\\ 0&I_l&0\end{pmatrix} & \text{when}\; n=2l+1.
\end{array}\right.
\end{equation}
The above form $\beta$ exists on every field and the form is unique up to equivalence and is called the \textit{split form} (see~\cite[Chapter 2]{Bh}).
\begin{definition}
The {\normalfont orthogonal group} is defined as 
\[ \mathrm{O}(n,k):=\{ g \in \GL(n,k) \mid \tr g\beta g=\beta \}, \text{ where } \beta \text{ as in Equation }\eqref{beta2}.\]
\end{definition}
\begin{definition}
	The {\normalfont orthogonal similitude group} with respect to $\beta$ (as in Equation~\eqref{beta2}) is defined by
	\[\GO(n,k)=\{g\in \GL(n,k) \mid \tr g \beta g=\mu(g)\beta, \text{ for some }\; \mu(g) \in k^{\times} \},\] 
	where $\mu : \GO(n,k)\rightarrow k^{\times}$ 
	is a group homomorphism with $\mathrm{ker}\;\mu=\mathrm{O}(n,k)$ and the factor $\mu(g)$ is called the \textit{multiplier} of $g$.
\end{definition}
Next, we define the twisted analog of the orthogonal group. 
We talk about twisted form only when $k=\mathbb{F}_q$, a finite field. For the twisted form, we fix a basis $\{e_1, e_{-1}, e_2,\ldots, e_l, e_{-2},\ldots, e_{-l}\}$  so that the matrix $\beta$ is: 
\begin{equation}\label{twisted_beta}
\beta=\begin{pmatrix}
\beta_{0}&0&0\\
0&0&I_{l-1}\\
0&I_{l-1}&0
\end{pmatrix}
\end{equation} where 
$\beta_{0}=\begin{pmatrix}
1&0\\
0&\epsilon
\end{pmatrix} $ and $\epsilon$ is a fixed non-square in $k^{\times}$, i.e., $\epsilon \in \mathbb{F}_q^{\times}\setminus\mathbb{F}_q^{\times 2}$.
\begin{definition}\label{to-}
The {\normalfont twisted orthogonal group} is defined as 
\[ \mathrm{O}^{-}(n,q)=\mathrm{O}^{-}(2l,q):=\{ g \in \GL(n,q) \mid \tr g\beta g=\beta \}, \text{ where } \beta \text{ as in Equation }(\ref{twisted_beta}).\]
\end{definition}
\begin{definition}
	The {\normalfont twisted orthogonal similitude group} with respect to $\beta$ (Equation~\eqref{twisted_beta}), is defined by
	\[\GO^{-}(n,q)=\{g\in \GL(n,q) \mid \tr g \beta g=\mu(g)\beta, \text{ for some }\; \mu(g) \in \mathbb{F}_q^{\times} \},\] 
	where $\mu : \GO^{-}(n,q)\rightarrow \mathbb{F}_q^{\times}$ 
	is a group homomorphism with $\mathrm{ker}\;\mu=\mathrm{O}^{-}(n,q)$. 
\end{definition}

\subsection{Clifford algebra}
Clifford algebras are far-reaching generalizations of the classical Hamiltonian quaternions. One motivation to study Clifford algebras comes from the Euclidean rotational groups.
For details, we refer to the reader \cite[Chapters 8 and 9]{gr}. 
Let $(V, \beta)$ be a quadratic space.
Let \[\mathrm{C}(V, \beta)=\frac{T(V)}{\langle x\tensor x-\beta(x, x).1 \mid x\in V\rangle}\]
be the {\it Clifford algebra}, where $T(V)$ is the tensor algebra.
Then $\mathrm{C}(V,\beta)$ is $\mathbb Z/2\mathbb Z$-graded algebra, say, $\mathrm{C}(V,\beta)=\mathrm{C}_0(V,\beta)\oplus \mathrm{C}_1(V,\beta)$.
The subalgebra $C_0(V,\beta)$ is called {\it special Clifford algebra} 
and it is a Clifford algebra in its own right. Then there is a unique anti-automorphism, say $\alpha : \mathrm{C}(V, \beta) \rightarrow \mathrm{C(V, \beta)}$ such that $\alpha|_{V}=Id_V$ (see, for example, \cite[Proposition 8.15.]{gr}). Now suppose that $u_1, u_2, \ldots, u_m$ are non-zero anisotropic vectors in $V$ such that $\rho_{u_1}\rho_{u_2}\cdots\rho_{u_m}=Id_V$, then $\prod_{i=1}^m \frac{1}{2}\beta(u_i, u_i)\in k^{\times 2}$ (for details, see \cite[Proposition 9.1.]{gr}). So from the above, we get a well-defined map from the orthogonal group to $k^{\times}/k^{\times 2}$ using Cartan-Dieudonne theorem. This map is called the \textit{spinor norm} on orthogonal group. See the next section for precise definition.

\subsubsection{Spinor norm}\label{spinornorms}
It is well-known that $\mathrm{P\Omega}(V)$ is a  simple group if $V$ contains an isotropic vector and dimension of $V$ is at least $5$ but we do not know when $-I$ is an element of the commutator subgroup of the orthogonal group. Then the theory of spinor norm comes into play via  Clifford algebra, for example, see Artin~\cite[Chapter V, Page 193]{Ar} or L. C. Grove~\cite[Chapter 9, Page 76]{gr}. Here we use the theory of spinor norm developed by G. E. Wall~\cite{wa}. For details and the connection between Clifford algebras and Wall's approach, refer to a nice article by R. Lipschitz~\cite{lips}.

\subsubsection{Classical spinor norm}
The classical way to define the spinor norm is via Clifford algebras~\cite[Chapters 8 and 9]{gr}. For $u\in V$ with $\beta(u,u)\neq 0$, we define the \textit{reflection} $\rho_u$ in the hyperplane orthogonal to $u$
by $\rho_{u}(v)=v-2\frac{\beta(u,v)}{\beta(u,u)}u$, which is an element of the orthogonal group. We know  
that every element of the orthogonal group 
$\mathrm{O}(n,k)$ can be written as a product of at most $n$ reflections.
\begin{definition}
	The {\normalfont spinor norm} is a group homomorphism 
	$\Theta : \mathrm{O}(n,k)\rightarrow k^{\times}/k^{\times2}$ defined by  
	$\Theta(g):=\displaystyle \left(\prod_{i=1}^m \frac{1}{2}\beta(u_{i},u_i)\right)\cdot k^{\times2}$, 
	where $g=\rho_{u_1}\cdots \rho_{u_m}$ is written as a product of reflections.
\end{definition}
However, in practice, it is difficult to use the above definition to compute the spinor norm.
\subsubsection{Wall's spinor norm}
 Wall~\cite{wa}, Zassenhaus~\cite{za} and Hahn~\cite{ha} developed a theory to compute the spinor norm. So we now define the spinor norm using Wall's idea. For our exposition, we follow~\cite[Chapter 11]{ta}. For more details on spinor norm using Wall's idea, see Bhunia \cite[Chapter 4, Page 41]{Bh}.

Let $g$ be an element of the orthogonal group. Let $\tilde g=I-g$ and $V_g=\tilde g(V)$. Using $\beta$ we define Wall's bilinear form $[\ ,\ ]_g$ on $V_g$ as follows:
$$[u,v]_g=\beta(u,y), \text{ where }\ v=\tilde g(y).$$
This bilinear form satisfies following properties:
\begin{enumerate}
	\item $[u,v]_g+[v,u]_g=\beta(u,v)$ for all $u,v\in V_g$.
	\item $g$ is an isometry on $V_g$ with respect to $[\ ,\ ]_g$.
	\item $[v,u]_g=-[u,gv]_g$ for all $u,v\in V_g$.
	\item $[\ ,\ ]_g$ is non-degenerate. 
\end{enumerate}
Then the \textit{spinor norm} is 
$$\Theta(g)= \overline{\mathrm{disc}(V_g,[\ ,\ ]_g)}\;  \text{if}\; g\neq I$$ 
extended to $I$ by defining $\Theta(I)=\overline 1$.
An element $g$ is called \textit{regular} if $V_g$ is a non-degenerate subspace of $V$ with respect to the form $\beta$. 
Hahn~\cite[Proposition 2.1]{ha} proved that for a regular element $g$, the spinor norm is $\Theta(g)=\overline{\det(\tilde g|_{V_g})\text{disc}(V_g)}$. This gives,
\begin{proposition}\label{propospinor}
	\begin{enumerate}
		\item For a reflection $\rho_v$, $\Theta(\rho_v)=\frac{\overline{\beta(v,v)}}{2}$.
		\item For a unipotent element $g$ the spinor norm is trivial, i.e., $\Theta(g)=\overline{1}$.
	\end{enumerate}
\end{proposition}
\begin{proof}
\begin{enumerate}
\item Let $\rho_{v}$ be a reflection in the hyperplane orthogonal to $v$, i.e., $\rho_v\in \mathrm{O}(n,k)$.
Then $V_{\rho_{v}}=\langle v \rangle$, therefore 
$\overline{\mathrm{disc}(V_{\rho_v},[\ ,\ ]_{\rho_v})}=\overline{\mathrm{det}([v, v]_{\rho_{v}})}=\frac{\overline{\beta(v,v)}}{2}$. 
Hence $\Theta(\rho_{v})=\frac{\beta(v,v)}{2}k^{\times2}$. 
\item See \cite[Corollary 2.2]{ha} for proof.
\end{enumerate}
\end{proof}
In this direction, we show that the Gaussian elimination algorithm we develop outputs the spinor norm (see Section~\ref{spinornorm}).
%%%%%%%%%%%%%%%%%%%%%%%%%%%%%%%%%%%%%%%%%%%%%%%%%%%%%%%%%%%%%%%%%%%
\section{Solving the word problem in classical reductive groups}\label{wordproblem}
Let $G$ be a reductive linear algebraic group over $k$. Then using root datum, $G$ is of $4$-classical types $A_l, B_l, C_l$ and $D_l$, and $5$-exceptional types $G_2, F_4, E_6, E_7$ and $E_8$ respectively. In this paper, we solve the word problem for classical type groups. The groups that correspond to these types are: 
\begin{itemize}
\item ($A_l$-type): $\GL(l+1,k)$, 

\item ($B_l$-type): $\GO(2l+1,k)$,

\item ($C_l$-type): $\GSp(2l, k)$,

\item ($D_l$-type): $\GO(2l,k)$.
\end{itemize}
 Let $G=\GL(l+1,k)$ be a general linear group, then one has a well-known algorithm to solve the word problem -- the Gaussian elimination. One observes that the effect of multiplying an element of the general linear group by an elementary matrix (also known as elementary transvection) from left or right is a row or a column operation respectively. Using this algorithm one can start with any matrix $g\in\mathrm{GL}(l+1, k)$ and get $\mathrm{diag}(1,\ldots,1,\mathrm{det}(g))$. Thus writing $g$ as a product of elementary matrices and a diagonal matrix. One objective of this paper is to discuss a similar algorithm for symplectic and orthogonal similitude groups.  

We first describe the elementary matrices and the elementary operations for the symplectic and orthogonal similitude groups. These elementary operations are nothing but multiplication by elementary matrices from left and right respectively. The elementary matrices used here are nothing but the Steinberg generators which follows from the theory of Chevalley groups.
For simplicity, we will write the algorithm for $\GSp(2l,k), \GO(2l, k)$, $\GO(2l+1, k)$, and $\GO^{-}(2l, k)$ separately.

%%%%%%%%%%%%%%%%%%%%%%%%%%%%%%%%%%%%%%%%%%%%%%%%%%%%%%%%%%%%%%%%%%%
\subsection{Elementary matrices and elementary operations}
In what follows, the scalar ${\rm t}$ varies over the field $k$ and $n=2l$ or $n=2l+1$. 
We define $te_{i,j}$ as the $n\times n$ matrix with $t$ in the $(i,j)$ position and zero everywhere else, where $1\leq i,j \leq l$. 
We simply use $e_{i,j}$ to denote $1e_{i,j}$. 
We often use the well-known matrix identity 
$e_{i,j}e_{k,l}=\delta_{j,k}e_{i,l}$, where 
$\delta_{j,k}$ is the Kronecker delta.
For more details on elementary matrices, we refer ~\cite[Chapter 11]{ca}.
\subsubsection{Elementary matrices for $\GO(2l,k)$}\label{elementarymate}
We index rows and columns by $1, \ldots, l, -1, \ldots, -l$. 
The elementary matrices are defined as follows:
\begin{eqnarray*}
	x_{i,j}(t) =& I + t(e_{i,j} - e_{-j,-i}) & \text{for $i \neq j$},\\
	x_{i,-j}(t) =& I + t(e_{i,-j} - e_{j,-i})& \text{for $i < j$},\\
	x_{-i,j}(t) =& I + t(e_{-i,j} - e_{-j,i}) & \text{for $i < j$},\\
	w_l =& I - e_{l,l} - e_{-l,-l} - e_{l,-l} - e_{-l,l}. 
\end{eqnarray*}
We write $g \in \mathrm{GO}(2l,k)$ as 
$g=\begin{pmatrix} A&B\\C&D\end{pmatrix}$, where $A,B,C$ and $D$ are $l \times l$ matrices. As $\tr g\beta g=\mu(g)\beta$, then we have $\tr AC+\tr CA=0=\tr BD+\tr DB$ and $\tr AD+\tr CB=\mu(g)I_l$. Let us note the effect of multiplying $g$ by elementary matrices in the following table.
\begin{table}[ht]
	%\scriptsize
	\caption{The elementary operations for $\mathrm{GO}(2l, k)$} \label{Table 1} \vskip 0.2 cm
	\centering % used for centering table
	\begin{tabular}{|c c||c c |} % centered columns (4 columns)
		\hline %inserts double horizontal lines
		%&&&\\
		& Row operations & &Column operations  \\ [0.5ex]
		\hline %inserts single line
		ER1 & $i^{\text{th}} \mapsto i^{\text{th}}+t j^{\text{th}}$ row and & EC1 & $j^{\text{th}} \mapsto j^{\text{th}}+t i^{\text{th}}$ column and \\
		& $-j^{\text{th}} \mapsto -j^{\text{th}}-t (-i)^{\text{th}}$  row & &$-i^{\text{th}} \mapsto -i^{\text{th}}-t(-j)^{\text{th}}$  column\\
		\hline
		ER2 & $i^{\text{th}} \mapsto i^{\text{th}}+t (-j)^{\text{th}}$ row and & EC2 & $-i^{\text{th}} \mapsto -i^{\text{th}}-tj^{\text{th}}$ column and \\
		&$j^{\text{th}} \mapsto j^{\text{th}}-t (-i)^{\text{th}}$  row & &$-j^{\text{th}} \mapsto -j^{\text{th}}+ti^{\text{th}}$  column\\
		\hline
		ER3 &$-i^{\text{th}} \mapsto -i^{\text{th}}-tj^{\text{th}}$ row and & EC3 & $j^{\text{th}} \mapsto j^{\text{th}}+t(-i)^{\text{th}}$ column and \\
		&$-j^{\text{th}} \mapsto -j^{\text{th}}+ti^{\text{th}}$  row & & $i^{\text{th}} \mapsto i^{\text{th}}-t(-j)^{\text{th}}$  column\\
		\hline
	\end{tabular}
\end{table} 
\subsubsection{Elementary matrices for $\mathrm{GSp}(2l,k)$}\label{elementarymatsp}
We index rows and columns by $1, \ldots, l, -1, \ldots, -l$. 
The elementary matrices are as follows:
\begin{eqnarray*}
	x_{i,j}(t)=&I+t(e_{i,j}-e_{-j,-i}) &\text{for}\;  i\neq j,\\
	x_{i,-j}(t)=&I+t(e_{i,-j}+e_{j,-i})&\text{for}\;  i<j,\\ 
	x_{-i,j}(t)=&I+t(e_{-i,j}+e_{-j,i})& \text{for}\;  i<j,\\
	x_{i,-i}(t)=&I+te_{i,-i}, \\
	x_{-i,i}(t)=&I+te_{-i,i}, 
\end{eqnarray*}
 We write $g \in \mathrm{GSp}(2l,k)$ as 
$g=\begin{pmatrix} A&B\\C&D\end{pmatrix}$, where $A,B,C$ and $D$ are $l \times l$ matrices. Let us note the effect of multiplying $g$ by elementary matrices in the following table.
\begin{table}[ht]
\centering
		\caption{The elementary operations for $\mathrm{GSp}(2l,k)$} \label{Table 2} \vskip 0.2 cm
		\begin{tabular}{|cc||cc|}
			\hline
			&Row operations && Column operations\\
			\hline
			ER1&$i\textsuperscript{th}\mapsto i\textsuperscript{th} + tj\textsuperscript{th}$ row and 
			&EC1& $j\textsuperscript{th}\mapsto j\textsuperscript{th} + ti\textsuperscript{th}$ column and\\ &$-j\textsuperscript{th}\mapsto -j\textsuperscript{th} +t(-i)\textsuperscript{th}$ row && $-i\textsuperscript{th}\mapsto -i\textsuperscript{th} + t(-j)\textsuperscript{th}$ column \\
			\hline
			ER2&$i\textsuperscript{th}\mapsto i\textsuperscript{th} + t(-j)\textsuperscript{th}$ row and 
			&EC2 & $-i\textsuperscript{th}\mapsto -i\textsuperscript{th} + tj\textsuperscript{th}$ column and\\ & $j\textsuperscript{th}\mapsto j\textsuperscript{th} +t(-i)\textsuperscript{th}$ row && $-j\textsuperscript{th}\mapsto -j\textsuperscript{th} + ti\textsuperscript{th}$ column\\
			\hline
			ER3&$-i\textsuperscript{th}\mapsto -i\textsuperscript{th} + tj\textsuperscript{th}$ row and 
			&EC3& $j\textsuperscript{th}\mapsto j\textsuperscript{th} + t(-i)\textsuperscript{th}$ column and\\ & $-j\textsuperscript{th}\mapsto -j\textsuperscript{th} +ti\textsuperscript{th}$ row && $i\textsuperscript{th}\mapsto i\textsuperscript{th} + t(-j)\textsuperscript{th}$ column\\
			\hline
			ER1a&$i\textsuperscript{th}\mapsto i\textsuperscript{th} + t(-i)\textsuperscript{th}$ row &EC1a& $-i\textsuperscript{th}\mapsto -i\textsuperscript{th}+ti\textsuperscript{th}$ column\\
			\hline
			ER2a&$-i\textsuperscript{th}\mapsto -i\textsuperscript{th} + ti\textsuperscript{th}$ row &EC2a& $i\textsuperscript{th}\mapsto i\textsuperscript{th}+t(-i)\textsuperscript{th}$ column\\
			\hline
		\end{tabular}
\end{table}

\subsubsection{Elementary matrices for $\GO(2l + 1, k)$}\label{elementarymato}
We index rows and columns by $0, 1,\ldots , l, -1, \ldots, -l$. 
The elementary matrices are defined as follows:
\begin{eqnarray*}
	x_{i,j}(t) =& I + t(e_{i,j} - e_{-j,-i} )& \text{for $i \neq j$},\\
	x_{i,-j}(t) =& I + t(e_{i,-j} - e_{j,-i})& \text{for $i < j$},\\
	x_{-i,j}(t)=& I + t(e_{-i,j} - e_{-j,i})& \text{for $i < j$},\\
	x_{i,0}(t) =& I + t(2e_{i,0} - e_{0,-i} ) - t^2 e_{i,-i}, \\
	x_{0,i}(t) =& I + t(-2e_{-i,0} + e_{0,i}) - t^2 e_{-i,i}, \\
	w_l =& I - e_{l,l} - e_{-l,-l} - e_{l,-l} - e_{-l,l}. 
\end{eqnarray*}
We write an element $g\in \mathrm{GO}(2l+1,k)$ as $g=\begin{pmatrix}\alpha&X&Y\\ E& A&B\\F&C&D\end{pmatrix}$, where $A,B,C$ and $D$ are $l\times l$ matrices, $X$ and $Y$ are $1\times l$ matrices, $E$ and $F$ are $l\times 1$ matrices, $\alpha\in k$. 
Let us note the effect of multiplying $g$ by elementary matrices in the following table.
\begin{table}[ht]
	%\scriptsize
	\caption{The elementary operations for $\mathrm{GO}(2l+1,k)$} \label{Table 3} \vskip 0.2 cm
	\centering % used for centering table
	\begin{tabular}{|c c ||c c|} % centered columns (4 columns)
		\hline %inserts double horizontal lines
		& Row operations & &Column operations  \\ [0.5ex]
		\hline %inserts single line
		ER1 & $i^{\text{th}} \mapsto i^{\text{th}}+t j^{\text{th}}$ row and & EC1 & $j^{\text{th}} \mapsto j^{\text{th}}+t i^{\text{th}}$ column and \\
		& $-j^{\text{th}} \mapsto -j^{\text{th}}-t (-i)^{\text{th}}$  row &  &$-i^{\text{th}} \mapsto -i^{\text{th}}-t(-j)^{\text{th}}$  column\\
		\hline
		ER2 & $i^{\text{th}} \mapsto i^{\text{th}}+t (-j)^{\text{th}}$ row and & EC2 & $-i^{\text{th}} \mapsto -i^{\text{th}}-tj^{\text{th}}$ column and \\
		&$j^{\text{th}} \mapsto j^{\text{th}}-t (-i)^{\text{th}}$  row & &$-j^{\text{th}} \mapsto -j^{\text{th}}+ti^{\text{th}}$  column\\
		\hline
		ER3 &$-i^{\text{th}} \mapsto -i^{\text{th}}-tj^{\text{th}}$ row and & EC3 & $j^{\text{th}} \mapsto j^{\text{th}}+t(-i)^{\text{th}}$ column and \\
		&$-j^{\text{th}} \mapsto -j^{\text{th}}+ti^{\text{th}}$  row & & $i^{\text{th}} \mapsto i^{\text{th}}-t(-j)^{\text{th}}$  column\\
		\hline
		ER4 & $0^{\text{th}} \mapsto 0^{\text{th}}-t(-i)^{\text{th}}$ row and & EC4 & $0^{\text{th}} \mapsto 0^{\text{th}}+2ti^{\text{th}}$ column and \\
		&$i^{\text{th}} \mapsto i^{\text{th}}+2t0^{\text{th}}-t^{2}(-i)^{\text{th}}$  row & & $(-i)^{\text{th}} \mapsto (-i)^{\text{th}}-t0^{\text{th}}-t^{2}i^{\text{th}}$  column\\
		\hline
		ER4a &$0^{\text{th}} \mapsto 0^{\text{th}}+ti^{\text{th}}$ row and & EC4a & $0^{\text{th}} \mapsto 0^{\text{th}}-2t(-i)^{\text{th}}$ column and \\
		& $(-i)^{\text{th}} \mapsto (-i)^{\text{th}}-2t0^{\text{th}}-t^{2}i^{\text{th}}$  row 
		& &$i^{\text{th}} \mapsto i^{\text{th}}+t0^{\text{th}}-t^{2}(-i)^{\text{th}}$  column\\
		\hline
	\end{tabular}
\end{table}

\subsection{Gaussian elimination}
To explain the steps of our Gaussian elimination algorithm we need some lemmas.
In this subsection, we prove these lemmas.
\begin{lemma}\label{lemma1}
	Let $Y=\mathrm{diag}(1,\ldots,1,\lambda,\ldots,\lambda)$ be of size $l$
	with the number of $1$s equal to $m<l$. Let $X$ be a matrix of size $l$ such that 
	$YX$ is symmetric (resp. skew-symmetric) then $X$ is of the form 
	$ \begin{pmatrix}X_{11}&X_{12}\\X_{21}&X_{22}\end{pmatrix}$, 
	where $X_{11}$ is an $m\times m$ symmetric (resp. skew-symmetric), and 
	$X_{12}=\lambda \tr X_{21}$ (resp. $X_{12}=-\lambda \tr X_{21}$).
	Furthermore, if $\lambda \neq 0$ then $X_{22}$ is symmetric (resp. skew-symmetric).
\end{lemma}
\begin{proof}
	First, observe that the matrix $YX=\begin{pmatrix}X_{11}&X_{12}\\\lambda X_{21}&\lambda X_{22}\end{pmatrix}$.
	Since the matrix $YX$ is symmetric (resp. skew-symmetric), then $X_{11}$ is symmetric (resp. skew-symmetric),
	and $X_{12}=\lambda \tr X_{21}$ (resp. $X_{12}=-\lambda \tr X_{21}$).
	Also if $\lambda \neq 0$ then $X_{22}$ is symmetric (resp. skew-symmetric).
\end{proof}
\begin{corollary}\label{lemma2}
	Let $g=\begin{pmatrix} A&B\\C&D \end{pmatrix}$ be either in $\GSp(2l,k)$ or $\GO(2l,k)$.
	\begin{enumerate}
		\item If $A$ is a diagonal matrix $\mathrm{diag}(1,\ldots,1,\lambda), \lambda \in k^{\times}$,  
		then the matrix $C$ is of the form $\begin{pmatrix}C_{11}&\pm \lambda \tr C_{21} 
		\\C_{21} & c_{ll}\end{pmatrix}$, 
		where $C_{11}$ is an $(l-1)\times(l-1)$ symmetric if $g \in \GSp(2l,k)$,  
		and $C_{11}$ is skew-symmetric with $c_{ll}=0$ if $g \in \GO(2l,k)$.
		\item If $A$ is a diagonal matrix $\mathrm{diag}(\underbrace{1,\ldots,1}_{m},\underbrace{0,\ldots,0}_{l-m})$, 
		then the matrix $C$ is of the form 
		$\begin{pmatrix} C_{11} &0\\C_{21}&C_{22}\end{pmatrix}$, where $C_{11}$ is an $m\times m$
		symmetric matrix if $g\in \GSp(2l,k)$, and is skew-symmetric if $g\in \GO(2l,k)$.
	\end{enumerate}
\end{corollary}
\begin{proof}
	We use the condition that $g$ satisfies $\tr g \beta g=\mu(g)\beta$, and 
	$AC$ is symmetric (using $\tr A=A$, as $A$ is diagonal)  
	when $g\in \GSp(2l,k)$, and $AC$ is skew-symmetric when 
	$g\in \GO(2l,k)$. Then Lemma~\ref{lemma1} gives the required 
	form for $C$.
\end{proof}
\begin{corollary}\label{lemma3}
	Let $g=\begin{pmatrix}A&B\\0&\mu(g)A^{-1}\end{pmatrix} \in \GSp(2l,k)$ or $\GO(2l,k)$, where 
	$A=\mathrm{diag}(1,\ldots,1,\lambda)$, then the matrix $B$ is of the form 
	$\begin{pmatrix} B_{11}&\pm \lambda^{-1}\tr B_{21}\\B_{21}&b_{ll}\end{pmatrix}$, 
	where $B_{11}$ is a symmetric matrix of size $l-1$ if $g\in \GSp(2l,k)$, and 
	skew-symmetric with $b_{ll}=0$ if $g\in \GO(2l,k)$.
\end{corollary}
\begin{proof}
	We use the condition that $g$ satisfies $\tr g \beta g=\mu(g)\beta$ and 
	$\tr A=A$ to get $A^{-1}B$ is symmetric if $g\in \GSp(2l,k)$, and skew-symmetric 
	if $g \in \GO(2l,k)$. Again Lemma~\ref{lemma1} gives the required form 
	for $B$.
\end{proof}
\begin{lemma}\label{lemma4}
	Let $g=\begin{pmatrix} A&B\\0&D\end{pmatrix}\in \GL(2l,k)$. Then,
	\begin{enumerate}
		\item  $g \in \GSp(2l,k)$ if and only if 
		$D=\mu(g)\tr A^{-1}$ and $\tr (A^{-1}B)=(A^{-1}B)$, and 
		\item  $g \in \GO(2l,k)$ if and only if 
		$D=\mu(g)\tr A^{-1}$ and $\tr (A^{-1}B)=-(A^{-1}B)$.
	\end{enumerate}
\end{lemma}
\begin{proof}
	\begin{enumerate}
		\item Let $g\in \GSp(2l,k)$ then $g$ satisfies $\tr g \beta g=\mu(g)\beta$.
		Then this implies $D=\mu(g)\tr A^{-1}$ and $\tr (A^{-1}B)=(A^{-1}B)$. 
		
		Conversely, if $g$ satisfies the given condition then clearly $g\in \GSp(2l,k)$.
		\item This follows by similar computation.
	\end{enumerate}
\end{proof}
\begin{lemma}\label{lemma5}
	Let $Y=\mathrm{diag}(1,\ldots,1,\lambda)$ be of size 
	$l$, where $\lambda \in k^{\times}$ and $X=(x_{ij})$ be 
	a matrix such that $YX$ is symmetric (resp. skew-symmetric).
	Then $X=(R_{1}+R_{2}+\ldots)Y$, where each $R_{m}$ is of the 
	form $t(e_{i,j}+e_{j,i})$ for some $i<j$ or of the form $te_{i,i}$
	for some $i$ (resp. each $R_{m}$ is of the form 
	$t(e_{i,j}-e_{j,i})$ for some $i<j$).
\end{lemma}
\begin{proof}
	Since the matrix $YX$ is symmetric (resp. skew-symmetric), 
	then the matrix $X$ is of the form 
	$\begin{pmatrix}X_{11}&X_{12}\\X_{21}&x_{ll}\end{pmatrix}$, 
	where $X_{11}$ is symmetric (resp. skew-symmetric), 
	$X_{12}=\lambda \tr X_{21}$ (resp. $X_{12}=-\lambda \tr X_{21}$) 
	and $X_{21}$ is a row of size $l-1$. 
	Clearly, $X$ is a sum of the matrices of the form 
	$R_m Y$.    
\end{proof}
\begin{lemma}\label{lemma6}
	For $1\leq i\leq l$, 
	\begin{enumerate}
		\item The element $w_{i,-i}=I+e_{i,-i}-e_{-i,i}-e_{i,i}-e_{-i,-i} \in \GSp(2l,k)$ 
		is a product of elementary matrices.
		\item The element $w_{i,-i}=I-e_{i,-i}-e_{-i,i}-e_{i,i}-e_{-i,-i} \in \GO(2l,k)$ 
		is a product of elementary matrices.
		\item The element $w_{i,-i}=I-2e_{0,0}-e_{i,-i}-e_{-i,i}-e_{i,i}-e_{-i,-i} \in \GO(2l+1,k)$ 
		is a product of elementary matrices.     
	\end{enumerate}
\end{lemma}
\begin{proof}
	\begin{enumerate}
		\item We have $w_{i,-i}=x_{i,-i}(1)x_{-i,i}(-1)x_{i,-i}(1)$.
		\item We produce these elements inductively. 
		First we get $w_{i,-j}=(I+e_{i,-j}-e_{j,-i})(I+e_{-i,j}-e_{-j,i})(I+e_{i,-j}-e_{j,-i})
		=x_{i,-j}(1)x_{-i,j}(1)x_{i,-j}(1)$, and 
		$w_{i,j}=(I+e_{i,j}-e_{-j,-i})(I-e_{j,i}+e_{-i,-j})(I+e_{i,j}-e_{-j,-i})
		=x_{i,j}(1)x_{j,i}(-1)x_{i,j}(1)$. 
		Set $w_{l}:=w_{l,-l}=I-e_{l,l}-e_{-l,-l}-e_{l,-l}-e_{-l,l}$. 
		Then compute $w_{l}w_{l,l-1}w_{l,-(l-1)}=w_{(l-1),-(l-1)}$.
		So inductively we get $w_{i,-i}$ is a product of elementary matrices.
		\item We have $w_{i,-i}=x_{0,i}(-1)x_{i,0}(1)x_{0,i}(-1)$.
	\end{enumerate}
\end{proof}
\begin{lemma}\label{lemma7}
	The element $\mathrm{diag}(1,\ldots,1,\lambda,1,\ldots,1,\lambda^{-1})\in \GSp(2l,k)$ 
	is a product of elementary matrices.
\end{lemma}
\begin{proof}
	First we compute 
	\begin{align*}
		w_{l,-l}(t)&=(I+te_{l,-l})(I-t^{-1}e_{-l,l})(I+te_{l,-l}) \\
		&=I-e_{l,l}-e_{-l,-l}+te_{l,-l}-t^{-1}e_{-l,l} \\
		&=x_{l,-l}(t)x_{-l,l}(-t^{-1})x_{l,-l}(t).
	\end{align*}                             
	Then compute 
	\begin{align*}
		h_{l}(\lambda)&=w_{l,-l}(\lambda)w_{l,-l}(-1) \\
		&=I-e_{l,l}-e_{-l,-l}+\lambda e_{l,l}+\lambda^{-1}e_{-l,-l}, 
	\end{align*}                      
	which is the required element.
\end{proof}
\begin{lemma}\label{lemma8}
	Let $g=\begin{pmatrix} \alpha &X&Y\\E&A&B\\F&C&D\end{pmatrix} \in \GO(2l+1,k)$. Then, 
	\begin{enumerate}
		\item if $A=\mathrm{diag}(1,\ldots,1,\lambda)$ and $X=0$, then 
		$C$ is of the form $\begin{pmatrix}C_{11}&-\lambda \tr C_{21}\\C_{21}&0 \end{pmatrix}$ 
		with $C_{11}$ skew-symmetric.
		\item If $A=\mathrm{diag}(\underbrace{1,\ldots,1}_{m},\underbrace{0,\ldots,0}_{l-m})$, 
		and $X$ with its first $m$ entries $0$, then 
		$C$ is of the form $\begin{pmatrix}C_{11}&0\\C_{21}&C_{22} \end{pmatrix}$ 
		with $C_{11}$ skew-symmetric.
	\end{enumerate}
\end{lemma}
\begin{proof}
	We use the equation $\tr g\beta g=\mu(g)\beta$, and get 
	$2\tr XX+\tr AC+\tr CA=0$. In the first case, $AC$ is skew-symmetric 
	(using $X=0$ and $\tr A=A$). Then Lemma~\ref{lemma1} and Corollary~\ref{lemma2} 
	give the required form for $C$. In the second case, we note that $\tr XX$ has 
	top-left and top-right blocks $0$, and get the required form for $C$.
\end{proof}
\begin{lemma}\label{lemma9}
	Let $g=\begin{pmatrix}\alpha &X&Y\\E&A&B\\F&0&D \end{pmatrix} \in \GO(2l+1,k)$, then
	$X=0$, and $D=\mu(g)\tr A^{-1}$.
\end{lemma}
\begin{proof}
	We compute $\tr g\beta g=\mu(g)\beta$, and get $2\tr XX=0$ and 
	$2\tr XY+\tr AD=\mu(g)I$. Hence $X=0$, and $D=\mu(g)\tr A^{-1}$.
\end{proof}
\begin{lemma}\label{lemma10}
	Let $g=\begin{pmatrix}\alpha &0&Y\\0&A&B\\F&0&D \end{pmatrix}$, 
	with $A$ an invertible diagonal matrix. Then 
	$g\in \GO(2l+1,k)$ if and only if $\alpha^2=\mu(g), F=0=Y, D=\mu(g)A^{-1}$ and 
	$\tr DB+\tr BD=0$, where $\mu(g) \in k^{\times}$ is the multiplier of $g$.
\end{lemma}
\begin{proof}
	Let $g\in \GO(2l+1,k)$ then we have $\tr g\beta g=\mu(g)\beta$. So we get 
	$\alpha^2=\mu(g), F=0=Y, D=\mu(g)A^{-1}$ and $\tr DB+\tr BD=0$.
	
	Conversely, if $g$ satisfies the given condition, then $g \in \GO(2l+1,k)$.
\end{proof}
We are now in a position to describe our algorithms.
\subsubsection{Gaussian elimination for $\GO(2l,k)$ and  $\GSp(2l,k)$ }\label{gausseven}
The algorithm is as follows:

\vspace{2.5mm}
\noindent
Step $1$: 
\vspace{-8mm}
	\begin{enumerate}
	\item[] \textit{Input}: A matrix $g=\begin{pmatrix}A&B\\C&D \end{pmatrix} \in \GSp(2l,k)$ or $\GO(2l,k)$.\\
	\item[] \textit{Output}: The matrix $g_1=\begin{pmatrix}A_1&B_1\\C_1&D_1 \end{pmatrix}$ 
	is one of the following kind:\\
	\begin{enumerate}
		\item The matrix $A_1$ is a diagonal matrix 
		$\mathrm{diag}(1,\ldots,1,\lambda)$ with $\lambda \neq 0$, 
		and $C_1=\begin{pmatrix}C_{11}&C_{12}\\C_{21}&c_{ll}\end{pmatrix}$,  
		where $C_{11}$ is symmetric, when $g\in \GSp(2l,k)$, and 
		skew-symmetric, when $g\in \GO(2l,k)$, and is of size $l-1$. 
		Furthermore, $C_{12}=\lambda \tr C_{21}$, when $g \in \GSp(2l,k)$, and 
		$C_{12}=-\lambda \tr C_{21}, c_{ll}=0$, when $g \in \GO(2l,k)$.
		\item The matrix $A_1$ is a diagonal matrix $\mathrm{diag}
		(\underbrace{1,\ldots,1}_{m},\underbrace{0,\ldots,0}_{l-m})$, and 
		$C_1=\begin{pmatrix} C_{11}&0\\C_{21}&C_{22}\end{pmatrix}$, where 
		$C_{11}$ is an $m\times m$ symmetric, when $g\in \GSp(2l,k)$ and 
		skew-symmetric, when $g\in \GO(2l,k)$.
	\end{enumerate}
	\item[] \textit{Justification}: Observe the effect of ER$1$ and EC$1$ on the block $A$. 
	This amounts to the classical Gaussian elimination on a $l\times l$ matrix $A$. 
	Thus we can reduce $A$ to a diagonal matrix, and Corollary~\ref{lemma2} makes sure 
	that $C$ has the required form. 
\end{enumerate} 
\vspace{1mm}    
Step $2$: 
\vspace{-8mm}  
\begin{enumerate}
	\item[] \textit{Input}: matrix $g_1=\begin{pmatrix}A_1&B_1\\C_1&D_1\end{pmatrix}$.
	\item[] \textit{Output}: matrix $g_2=\begin{pmatrix}A_2&B_2\\0&\mu(g) \tr A_{2}^{-1}\end{pmatrix}
	;\; A_2=\mathrm{diag}(1,\ldots,1,\lambda)$.
	\item[] \textit{Justification}: Observe the effect of ER$3$. It changes $C_1$ by 
	$RA_1+C_1$. Using Lemma~\ref{lemma5} we can make the matrix $C_1$ the zero matrix 
	in the first case, and $C_{11}$ the zero matrix in the second case. Furthermore, in the second 
	case, we make use of Lemma~\ref{lemma6} to interchange the rows, so that we get a zero matrix 
	in place of $C_1$. If required, use ER$1$ and EC$1$ to make $A_1$ a diagonal matrix.
	Lemma~\ref{lemma4} ensures that $D_1$ becomes $\mu(g) \tr A_2^{-1}$.
\end{enumerate}
Step $3$: 
\vspace{-8mm}
\begin{enumerate}
	\item[] \textit{Input}: matrix $g_2=\begin{pmatrix}A_2&B_2\\0&\mu(g)\tr A_2^{-1}\end{pmatrix}
	;\; A_2=\mathrm{diag}(1,\ldots,1,\lambda)$.
	\item[] \textit{Output}: matrix $g_3=\mathrm{diag}(1,\ldots,1,\lambda,\mu(g),\ldots,\mu(g),\mu(g) \lambda^{-1})$.
	\item[] \textit{Justification}: Using Corollary~\ref{lemma3} we see that the matrix 
	$B_2$ has a certain form. We can use ER$2$ to make the matrix $B_2$ a zero matrix 
	because of Lemma~\ref{lemma5}.
\end{enumerate}
The algorithm terminates here for $\GO(2l,k)$. However, for $\GSp(2l,k)$ there is one more step.\\

\noindent
Step $4$:
\vspace{-7.0mm}
\begin{enumerate}
	\item[] \textit{Input}: matrix $g_3=\mathrm{diag}(1,\ldots,1,\lambda,\mu(g),\ldots,\mu(g),\mu(g) \lambda^{-1})$.
	\item[] \textit{Output}: matrix $g_4=\mathrm{diag}(1,\ldots,1,\mu(g),\ldots,\mu(g))$, where $\mu(g)\in k^{\times}$.
	\item[] \textit{Justification}: Using Lemma~\ref{lemma7}.
\end{enumerate}
\subsubsection{Gaussian elimination for $\GO(2l+1,k)$}\label{gaussodd}
The algorithm is as follows:\\

\noindent
\vspace{3mm}
Step $1$:
\vspace{-14.2mm}
\begin{enumerate}
	\item[] \textit{Input}: A matrix $g=\begin{pmatrix}\alpha &X&Y\\E&A&B\\F&C&D\end{pmatrix} \in \GO(2l+1,k)$.
	\item[] \textit{Output}: The matrix $g_1=\begin{pmatrix}\alpha_1 &X_1&Y_1\\E_1&A_1&B_1\\F_1&C_1&D_1\end{pmatrix}$ 
	is one of the following kind:
	\begin{enumerate}
		\item The matrix $A_1$ is a diagonal matrix $\mathrm{diag}(1,\ldots,1,\lambda)$ with $\lambda \neq 0$.
		\item The matrix $A_1$ is a diagonal matrix $\mathrm{diag}(\underbrace{1,\ldots,1}_{m},\underbrace{0,\ldots,0}_{l-m}) (m<l)$.
	\end{enumerate}
	\item[] \textit{Justification}: Using ER$1$ and EC$1$, we do the 
	classical Gaussian elimination on a 
	$l\times l$ matrix $A$.
\end{enumerate}
\vspace{3.0mm}
Step $2$:
\vspace{-10.5mm}
\begin{enumerate}
	\item[] \textit{Input}: matrix $g_1=\begin{pmatrix}\alpha_1 &X_1&Y_1\\E_1&A_1&B_1\\F_1&C_1&D_1\end{pmatrix}$.
	\item[] \textit{Output}: matrix $g_2=\begin{pmatrix}\alpha_2 &X_2&Y_2\\E_2&A_2&B_2\\F_2&C_2&D_2\end{pmatrix}$ 
	is one of the following kind:
	\begin{enumerate}
		\item The matrix $A_2$ is $\mathrm{diag}(1,\ldots,1,\lambda)$ with $\lambda \neq 0, X_2=0=E_2$, and 
		$C_2=\begin{pmatrix}C_{11}&-\lambda \tr C_{21}\\C_{21}&0\end{pmatrix}$, where $C_{11}$ is skew-symmetric of size 
		$l-1$.
		\item The matrix $A_2$ is $\mathrm{diag}(\underbrace{1,\ldots,1}_{m},\underbrace{0,\ldots,0}_{l-m}) (m<l)$; 
		$X_2, E_2$ have first $m$ entries $0$, and $C_2=\begin{pmatrix}C_{11}&0\\C_{21}&C_{22}\end{pmatrix}$, 
		where $C_{11}$ is an $m\times m$ skew-symmetric matrix.
	\end{enumerate}
	
	\item[] \textit{Justification}: Once we have $A_1$ in diagonal form, we use ER$4$ and EC$4$ to change 
	$X_1$ and $E_1$ to the required form. Then Lemma~\ref{lemma8} makes sure that $C_1$ has the required form.
\end{enumerate}
\vspace{2mm}
\noindent
Step $3$:
\vspace{-10.3mm}
\begin{enumerate}
	\item[] \textit{Input}: matrix $g_2=\begin{pmatrix}\alpha_2 &X_2&Y_2\\E_2&A_2&B_2\\F_2&C_2&D_2\end{pmatrix}$.
	\item[] \textit{Output}: 
	\begin{enumerate}
		\item matrix $g_3=\begin{pmatrix}\alpha_3 &0&Y_3\\0&A_3&B_3\\F_3&0&D_3\end{pmatrix};\quad A_3=
		\mathrm{diag}(1,\ldots,1,\lambda)$.
		\item matrix $g_3=\begin{pmatrix}\alpha_3 &X_3&Y_3\\E_3&A_3&B_3\\F_3&C_3&D_3\end{pmatrix};
		\quad  A_3=\mathrm{diag}(\underbrace{1,\ldots,1}_{m},\underbrace{0,\ldots,0}_{l-m})$; $X_3, E_3$ have first 
		$m$ entries $0$, and $C_3=\begin{pmatrix}0&0\\C_{21}&C_{22}\end{pmatrix}$.
	\end{enumerate}
	\item[] \textit{Justification}: Observe the effect of ER$3$, and Lemma~\ref{lemma5} ensures the 
	required form.
\end{enumerate}
\vspace{2.0mm}
\noindent
Step $4$:
\vspace{-11.0mm}
\begin{enumerate}
	\item[] \textit{Input}: matrix $g_3=\begin{pmatrix}\alpha_3 &X_3&Y_3\\E_3&A_3&B_3\\F_3&C_3&D_3\end{pmatrix}$.
	\item[] \textit{Output}: matrix $g_4=\begin{pmatrix}\alpha_4&0&0\\0&A_4&B_4\\0&0&\mu(g)A_4^{-1}\end{pmatrix}$ with 
	$A_4=\mathrm{diag}(1,\ldots,1,\lambda), \alpha_4^2=\mu(g)$, and 
	$B_4A_4+A_4\tr B_4=0$.
	
	\item[] \textit{Justification}: In the first case, Lemma~\ref{lemma10} ensures the required form. 
	In the second case, we interchange $i$ with $-i$ for $m+1\leq i \leq l$. This will make $C_3=0$.
	Then, if needed, we use ER$1$ and EC$1$ on $A_3$ to make it diagonal. Then Lemma~\ref{lemma9} 
	ensures that $A_3$ has full rank. Further, we can use ER$4$ and EC$4$ to make $X_3=0=E_3$. 
	Lemma~\ref{lemma10} gives the required form.
\end{enumerate}
\vspace{2.2mm}
\noindent
Step $5$:
\vspace{-11mm}
\begin{enumerate}
	\item[] \textit{Input}: matrix $g_4=\begin{pmatrix}\alpha_4&0&0\\0&A_4&B_4\\0&0&\mu(g)A_4^{-1}\end{pmatrix}; \quad 
	A_4=\mathrm{diag}(1,\ldots,1,\lambda), \alpha_4^2=\mu(g)$.
	\item[] \textit{Output}: matrix $g_5=\mathrm{diag}(\alpha_5,1,\ldots,1,\lambda,\mu(g),\ldots,\mu(g),\mu(g)\lambda^{-1})$
	with $\alpha_5^2=\mu(g)$.
	\item[] \textit{Justification}: Lemma~\ref{lemma10} ensures that $B_4$ is of a certain kind. We 
	can use ER$2$ to make $B_4=0$. 
\end{enumerate}

%%%%%%%%%%%%%%%%%%%%%%%%%%%%%%%%%%%%%%%%%%%%%%%%%%%%%%%%%%%%%%%%%
\subsection{Gaussian elimination algorithm for twisted orthogonal similitude groups}
Over a finite field $k=\mathbb{F}_q$, there are two types of even dimensional orthogonal groups. We covered one type in the previous section. In this section, we consider the other types. 
The algorithm, for the twisted orthogonal similitude group, is similar to the odd orthogonal similitude case, so we will be very brief here.
\subsubsection{Elementary matrices for $\GO^{-}(2l,q)$}\label{elementarymattwist}
We index rows and columns by $1, -1,2,\ldots , l, -2, \ldots, -l$. 
The elementary matrices are defined as follows:
\begin{eqnarray*}
x_{i,j}(t)=& I + t(e_{i,j} - e_{-j,-i})& \text{for $i \neq j$},\\
x_{i,-j}(t) =& I + t(e_{i,-j} - e_{j,-i})& \text{for $i < j$},\\
x_{-i,j}(t) =& I + t(e_{-i,j} - e_{-j,i})& \text{for $i < j$},\\
x_{i,1}(t) =& I + t(-2e_{-i,1}+e_{1,i}) - t^2 e_{-i,i}, \\
x_{1,i}(t) =& I + t(2e_{i,1}-e_{1,-i}) - t^2 e_{i,-i}, \\
x_{i,-1}(t) =& I + t(-2\epsilon e_{-i,-1}+e_{-1,i}) -\epsilon t^2 e_{-i,i}, \\
x_{-1,i}(t)=& I + t(2\epsilon e_{i,-1}-e_{-1,-i}) -\epsilon t^2 e_{i,-i}, \\
w_i =& I - e_{i,i} - e_{-i,-i} - e_{i,-i} - e_{-i,i},\\ 
x_1(t,s)=& I+(t-1)e_{1,1}-(t+1)e_{-1,-1}+s(e_{-1,1}+\epsilon e_{1,-1}),\\
x_2=& I-2e_{-1,-1}.
\end{eqnarray*}
Let $g=\begin{pmatrix}A_0 &X&Y\\E&A&B\\F&C&D\end{pmatrix} \in \GO^{-}(2l,q)$, where $A_0$ is a $2\times 2$ matrix, $X, Y$ are $2\times (l-1)$ matrix, $E, F$ are $(l-1)\times 2$ matrix and $A, B, C, D$ are $(l-1)\times (l-1)$ matrix respectively. Let us note the effect of multiplying $g$ by elementary matrices in the following table.  
\begin{table}[ht]
	%\scriptsize
	\caption{The elementary operations for $\mathrm{GO}^{-}(2l,q)$} \label{Table 4} \vskip 0.2 cm
	\centering % used for centering table
	\begin{tabular}{|c c ||c c|} % centered columns (4 columns)
		\hline %inserts double horizontal lines
		& Row operations & &Column operations  \\ [0.5ex]
		\hline %inserts single line
		ER1 & $i^{\text{th}} \mapsto i^{\text{th}}+t j^{\text{th}}$ row and & EC1 & $j^{\text{th}} \mapsto j^{\text{th}}+t i^{\text{th}}$ column and \\
		& $-j^{\text{th}} \mapsto -j^{\text{th}}-t (-i)^{\text{th}}$  row &  &$-i^{\text{th}} \mapsto -i^{\text{th}}-t(-j)^{\text{th}}$  column\\
		\hline
		ER2 & $i^{\text{th}} \mapsto i^{\text{th}}+t (-j)^{\text{th}}$ row and & EC2 & $-i^{\text{th}} \mapsto -i^{\text{th}}-tj^{\text{th}}$ column and \\
		&$j^{\text{th}} \mapsto j^{\text{th}}-t (-i)^{\text{th}}$  row & &$-j^{\text{th}} \mapsto -j^{\text{th}}+ti^{\text{th}}$  column\\
		\hline
		ER3 &$-i^{\text{th}} \mapsto -i^{\text{th}}-tj^{\text{th}}$ row and & EC3 & $j^{\text{th}} \mapsto j^{\text{th}}+t(-i)^{\text{th}}$ column and \\
		&$-j^{\text{th}} \mapsto -j^{\text{th}}+ti^{\text{th}}$  row & & $i^{\text{th}} \mapsto i^{\text{th}}-t(-j)^{\text{th}}$  column\\
		\hline
		ER4 & $1^{\text{st}} \mapsto 1^{\text{st}}-t(-i)^{\text{th}}$ row and & EC4 & $1^{\text{st}} \mapsto 1^{\text{st}}+2ti^{\text{th}}$ column and \\
		&$i^{\text{th}} \mapsto i^{\text{th}}+2t1^{\text{st}}-t^{2}(-i)^{\text{th}}$  row & & $-i^{\text{th}} \mapsto -i^{\text{th}}-t1^{\text{st}}-t^{2}i^{\text{th}}$  column\\
		\hline
		ER4a &$1^{\text{st}} \mapsto 1^{\text{st}}+ti^{\text{th}}$ row and & EC4a & $1^{\text{st}} \mapsto 1^{\text{st}}-2t(-i)^{\text{th}}$ column and \\
		& $-i^{\text{th}} \mapsto -i^{\text{th}}-2t1^{\text{st}}-t^{2}i^{\text{th}}$  row 
		& &$i^{\text{th}} \mapsto i^{\text{th}}+t1^{\text{st}}-t^{2}(-i)^{\text{th}}$  column\\
		\hline
		ER5 & $-1^{\text{th}} \mapsto -1^{\text{th}}-t(-i)^{\text{th}}$ row and & EC5 & $-1^{\text{th}} \mapsto -1^{\text{th}}+2\epsilon ti^{\text{th}}$ column and \\
		&$i^{\text{th}} \mapsto i^{\text{th}}+2\epsilon t(-1)^{\text{th}}-\epsilon t^{2}(-i)^{\text{th}}$  row & & $-i^{\text{th}} \mapsto -i^{\text{th}}-t(-1)^{\text{th}}-\epsilon t^{2}i^{\text{th}}$  column\\
		\hline
		ER5a &$-1^{\text{th}} \mapsto -1^{\text{th}}+ti^{\text{th}}$ row and & EC5a & $-1^{\text{th}} \mapsto -1^{\text{th}}-2\epsilon t(-i)^{\text{th}}$ column and \\
		& $-i^{\text{th}} \mapsto -i^{\text{th}}-2\epsilon t(-1)^{\text{th}}-\epsilon t^{2}i^{\text{th}}$  row 
		& &$i^{\text{th}} \mapsto i^{\text{th}}+t(-1)^{\text{th}}-\epsilon t^{2}(-i)^{\text{th}}$  column\\
		\hline
		$ w_{i}$ & Interchange $i^{th}$ and $(-i)^{th}$ row &$w_{i}$& Interchange $i^{th}$ and $(-i)^{th}$ column \\
		\hline
	\end{tabular}
\end{table}

The main reason the following algorithm works is the closed condition $\tr g\beta g=\mu(g)\beta$, which gives the following equations:
\begin{align}
\tr A_0\beta_0X+\tr EC+\tr FA &=0,\\
\tr A_0\beta_0Y+\tr ED+\tr FB &=0,\\
\tr X\beta_0X+\tr AC+\tr CA &=0,\\
\tr Y\beta_0Y+\tr BD+\tr DB &=0,
\end{align}
\begin{align}
\tr A_0\beta_0A_0+\tr EF+\tr FE &=\mu(g)\beta_0, \\
\tr X\beta_0Y+\tr AD+\tr CB &=\mu(g)I_{l-1}.
\end{align}
%%%%%%%%%%%%%%%%%%%%%%%%%%%%%%%%%%%%%%%%%%%%%%%%%%%%%%%%%%%%%%%%%%%
\subsubsection{Gaussian elimination algorithm for $\GO^{-}(2l,q)$}\label{gausstwist}
The algorithm is as follows:

\vspace{4mm}
\noindent
Step $1$: 
\vspace{-11mm}
\begin{enumerate}
	\item[] \textit{Input}: A matrix $g=\begin{pmatrix}A_0 &X&Y\\E&A&B\\F&C&D\end{pmatrix} \in \GO^{-}(2l,k)$.
	\item[] \textit{Output}: The matrix $g_1=\begin{pmatrix}(A_0)_1 &X_1&Y_1\\E_1&A_1&B_1\\F_1&C_1&D_1\end{pmatrix}$ 
	is one of the following kind:
	\begin{enumerate}
		\item The matrix $A_1$ is a diagonal matrix $\mathrm{diag}(1,\ldots,1,\lambda)$ with $\lambda \neq 0$.
		\item The matrix $A_1$ is a diagonal matrix $\mathrm{diag}(1,\ldots,1,0,\ldots,0)$ with number of $1$s
		equal to $m (<l-1)$.
	\end{enumerate}
	\item[] \textit{Justification}: Using ER$1$ and EC$1$ we do the classical Gaussian elimination on a 
	$(l-1)\times (l-1)$ matrix $A$.
\end{enumerate}
\vspace{3mm}
Step $2$: 
\vspace{-12.2mm}
\begin{enumerate}
	\item[] \textit{Input}: matrix $g_1=\begin{pmatrix}(A_0)_1 &X_1&Y_1\\E_1&A_1&B_1\\F_1&C_1&D_1\end{pmatrix}$.
	\item[] \textit{Output}: matrix $g_2=\begin{pmatrix}(A_0)_2 &0&Y_2\\E_2&A_2&B_2\\0&0&D_2\end{pmatrix}$, where $A_2=\mathrm{diag}\,(1,\ldots, 1, \lambda)$.
	
	\item[] \textit{Justification}: Observe the effect of ER$3$. In the first case, $C_1$ becomes zero matrix. In the second case, first interchange all zero rows of $A_1$ with the corresponding rows of $C_1$ using $w_i$.  This will make $C_1=0$. Then if needed use ER$1$ and EC$1$ on $A_1$ to make it diagonal. From the above equations we get $X_1=0$, $\tr AD=\mu(g)I_{l-1}$, and $F_1=0$ which ensures that $A_2$ has full rank. 
\end{enumerate}
\vspace{3.5mm}
Step $3$: 
\vspace{-11.5mm}
\begin{enumerate}
	\item[] \textit{Input}: matrix $g_2=\begin{pmatrix}(A_0)_2 &0&Y_2\\E_2&A_2&B_2\\0&0&\mu(g)A_2^{-1}\end{pmatrix}$.
	\item[] \textit{Output}: 
	matrix $g_3=\begin{pmatrix}(A_0)_3 &0&0\\0&A_3&B_3\\0&0&\mu(g)A_3^{-1}\end{pmatrix}$, where $A_3=
	\mathrm{diag}(1,\ldots,1,\lambda)$.
	
	\item[] \textit{Justification}: Use EC$4$ and EC$5$ to make $E_2=0$. Then from the above equation we get $Y_2=0$.
\end{enumerate}
\vspace{4mm}
Step $4$:
\vspace{-11.0mm}
\begin{enumerate}
	\item[] \textit{Input}: matrix $g_3=\begin{pmatrix}(A_0)_3 &0&0\\0&A_3&B_3\\0&0&\mu(g)A_3^{-1}\end{pmatrix}$
	\item[] \textit{Output}: matrix $g_4=\begin{pmatrix}(A_0)_4&0&0\\0&A_4&0\\0&0&\mu(g)A_4^{-1}\end{pmatrix}$, where 
	$A_4=\mathrm{diag}(1,\ldots,1,\lambda)$ and $(A_0)_4=\begin{pmatrix}t&\epsilon s\\s&-t\end{pmatrix}$ or $ \begin{pmatrix}t&-\epsilon s\\s&t\end{pmatrix}$ with $t^2+\epsilon s^2=\mu(g)=\mu(A_0)$.
	
	\item[] \textit{Justification}: We can use ER$2$ to make $B_3$ a zero matrix. If $(A_0)_4=\begin{pmatrix}t&\epsilon s\\s&-t\end{pmatrix}$, then the algorithm terminates here, otherwise go to the next step.
\end{enumerate}
\vspace{3mm}
Step $5$:
\vspace{-11mm}
\begin{enumerate}
	\item[] \textit{Input}: matrix $g_4=\begin{pmatrix}(A_0)_4&0&0\\0&A_4&0\\0&0&\mu(g)A_4^{-1}\end{pmatrix}$ with 
	$A_4=\mathrm{diag}(1,\ldots,1,\lambda)$ and $(A_0)_4=\begin{pmatrix}t&-\epsilon s\\s&t\end{pmatrix}$. 
	\item[] \textit{Output}: matrix $g_5=\mathrm{diag}(I_2,1,\ldots,1,\lambda,\mu(g),\ldots,\mu(g),\mu(g)\lambda^{-1})$.
	
	\item[] \textit{Justification}: Now using the elementary matrices $x_1(a,c)$ and $x_2$ we can reduce $g_4$ to the above form.	
\end{enumerate}
%%%%%%%%%%%%%%%%%%%%%%%%%%%%%%%%%%%%%%%%%%%%%%%%%%%%%%%%%%%%%%%%%%%
\subsection{Proof of~\ref{maintheorem}}\label{proofmaintheorem}
The proof follows from the above algorithms described in  Sections~\ref{gausseven},~\ref{gaussodd} and~\ref{gausstwist}. For word length, we mostly count the number of times elementary operations used. 
 We make $A$ a diagonal matrix by elementary operations. This has word length  $\mathcal{O}(l^2)$. In making both $B$ and $C$ a zero-matrix we  multiply two rows by a field
	element and additions. In the worst-case scenario, it has to be done with $\mathcal O(l)$ and $\mathcal O(l^2)$ many times. So the word length is $\mathcal{O}(l^3)$.
In the odd-orthogonal similitude group and twisted orthogonal similitude group we clear $X,Y,E,F$, so this has word length $\mathcal O(l^2)$. There are only a few steps that are independent of $l$. Then clearly, the word length is $\mathcal{O}(l^3)$.
\begin{remark}
The above algorithm works for groups defined over an arbitrary field not necessarily algebraically closed field. For example, 
	\begin{enumerate}
		\item Since all non-degenerate skew-symmetric bilinear forms are equivalent~\cite[Corollary 2.12]{gr}, we have a Gaussian elimination algorithm for all symplectic similitude groups over an arbitrary field. 
		\item Since non-degenerate symmetric bilinear forms over a finite field of odd characteristics are classified~\cite[Page 79]{gr} according to the $\beta$ (see Equations~\eqref{beta2} and~\eqref{twisted_beta}), we have a Gaussian elimination algorithm for all orthogonal similitude groups over a finite field of odd characteristics.
		\item Furthermore, we have a Gaussian elimination algorithm for orthogonal similitude groups that are given by the above bilinear form (see Equation~(\ref{beta2})) over an arbitrary field.
		\item For simplicity, we assume that  $\mathrm{char}\;(k)\neq 2$, though our algorithm works well on fields of all characteristics for symplectic and orthogonal similitude groups.
		Algorithms that we develop in this paper work only for a given bilinear form $\beta$ (see Equations (2.1)-(2.3)). Though in our algorithm, we work with only one bilinear form $\beta$, given by a fixed basis, with a suitable change of basis matrix our algorithm works for all  equivalent bilinear forms. 
	\end{enumerate}
	
\end{remark}
\section{Applications}
\subsection{Computing spinor norm for orthogonal groups}\label{spinornorm}
In this section, we show how we can use our Gaussian elimination algorithm to compute the spinor norm for orthogonal groups. Throughout this section, we assume that the field $k$ is of odd or zero characteristic. 
\begin{lemma}\label{spinornorm1}
For the group $\mathrm{O}(2l,k)$ or $\mathrm{O}(2l+1,k)$, we have 
 \begin{enumerate}
  \item $\Theta(x_{i,j}(t))=\Theta(x_{i,-j}(t))=\Theta(x_{-i,j}(t))=\overline{1}$. Furthermore, in the odd-dimensional case we also have $\Theta(x_{i,0}(t))=\Theta(x_{0,i}(t))=\overline 1$.
  \item $\Theta(w_l)=\overline{1}$.
  \item $\Theta(\diag(1,\ldots,1,\lambda,1,\ldots,1,\lambda^{-1}))=\overline{\lambda}$.
 \end{enumerate}
\end{lemma}
\begin{proof}
We use Proposition~\ref{propospinor}. First claim follows from the
fact that all the elementary matrices $x_{i,j}$  are unipotent. 
 Now the element $w_l=\rho_{(e_l+e_{-l})}$ is a reflection thus $\Theta(w_l)=\frac{1}{2}\overline{\beta(e_l+e_{-l},e_l+e_{-l})}=\overline{1}$.
 
For the third part we note that $\diag(1,\ldots,1,\lambda,1,\ldots,1,\lambda^{-1}) = \rho_{(e_l+e_{-l})} \rho_{(e_l+\lambda e_{-l})}$ and hence the spinor norm $\Theta(\diag(1,\ldots,1,\lambda,1,\ldots,1,\lambda^{-1}) )=\Theta(\rho_{(e_l+\lambda e_{-l})}) =\overline{\frac{\beta(e_l+\lambda e_{-l},e_l+\lambda e_{-l})}{2}}=\overline{\lambda}$.
\end{proof}
\subsection{Proof of~\ref{corollary}}\label{proofcorollary}
Let $g\in\mathrm{O}(2l,k)$ or $\mathrm{O}(2l+1,k)$. From \ref{maintheorem}, we write $g$ as a product of
elementary matrices and a diagonal matrix $\diag(1,\ldots,1,\lambda,1,\ldots,1,\lambda^{-1})$ and hence we can find the spinor norm of $g$ from Lemma~\ref{spinornorm}, i.e., $\Theta(g)=\bar{\lambda}$.

Now we compute the spinor norm in the twisted orthogonal group. First we observe the following:
\begin{lemma}\label{spinornormtwist} The spinor norm of elementary matrices in $\mathrm{O}^{-}(2l, q)$ are the following:
 \begin{enumerate}
  \item $\Theta(x_{i,j}(t))  = \Theta(x_{i,-j}(t)) = \Theta(x_{-i, j}(t))=\Theta(x_{i,1}(t))=\Theta(x_{1,i}(t))=\Theta(x_{i,-1}(t))=\Theta(x_{-1,i}(t)) =\bar{1}$.
  \item $\Theta(w_{i}) = \bar{1}$.
  \item $ \Theta(x_{1}(t,s))=\overline{(1-t)}$ \text{ whenever } $t\neq 1$.
  \item $\Theta(x_{2})=\overline{\frac{1}{2\epsilon}}$
  \item $\Theta(\diag(1,1,1,...,1,\lambda,1,...,1,\lambda^{-1})) = \bar{\lambda}$.
 \end{enumerate}
 \end{lemma}
 Proof: First one follows from Proposition~\ref{propospinor} as all the elementary matrices $x_{i,j}$ are unipotent. 
 The element $w_{i} = \rho_{e_{i}+e_{-i}}$ is a reflection thus 
 $\Theta(w_{i})= \overline{\frac{\beta(e_{i} +e_{-i},e_i+e_{-i})}{2}} = \bar{1}$.
 
For the third part, we observe that 
$x_{1}(t,s)=\rho_{(t-1)e_{1}+se_{-1}}$ and hence
\begin{align*}
\Theta(x_{1}(t,s))&=\Theta(\rho_{(t-1)e_{1}+se_{-1}})\\
&=\overline{\frac{(t-1)^2+\epsilon s^2}{2}} \\
&=\overline{(1-t)}  \quad \quad \text{as} \quad t^2+\epsilon s^2=1. 
\end{align*}
Note that $x_{2}=\rho_{e_{-1}}$ thus  $\Theta(x_{2})=\Theta(\rho_{e_{-1}})=\overline{\frac{\beta(e_{-1},e_{-1})}{2}}=\overline{\frac{\epsilon}{2}}$. 

For the last part, we note that
$\diag(1,1,1,...,1,\lambda,1,...,1,\lambda^{-1}) = \rho_{e_{l}+e_{-l}}\rho_{e_{l}+\lambda e_{-l}}$
and hence
\begin{align*}
\Theta(\diag(1,1,1,...,1,\lambda,1,...,1,\lambda^{-1}))& = \Theta(\rho_{e_{l}+e_{-l}})\Theta(\rho_{e_{l}+\lambda e_{-l}}) \\
&=\overline{\frac{{\beta(e_{l}+\lambda e_{-l},e_l+\lambda e_{-l})}}{2}}\\
&= \bar{\lambda}.
\end{align*}

\subsection{Proof of~\ref{corollary11}}\label{proofcorollary11}
Let $g\in \mathrm{O}^{-}(2l,q)$, then from \ref{maintheorem} we can write $g$ as a product of elementary matrices and a diagonal matrix $\diag(1,1,1,...,1,\lambda,1,...,1,\lambda^{-1})$. As the spinor norm is multiplicative, from Lemma~\ref{spinornormtwist} we get $\Theta(g)=\overline{\lambda(1-t)}$ if our algorithm (see Section~\ref{gausstwist}) terminates at Step 4, otherwise $\Theta(g)=\overline{\frac{\epsilon\lambda(1-t)}{2}}$.

%%%%%%%%%%%%%%%%%%%%%%%%%%%%%%%%%%%%%%%%%%%%%
\subsection{Double coset decomposition for Siegel maximal parabolic}\label{parabolicdecomp}

In this section, we compute the double coset decomposition with respect to Siegel maximal parabolic subgroup using our algorithm. Let $P$ be the Siegel maximal parabolic of $G$, where $G$ is either split orthogonal group $\mathrm{O}(n,k)$ or $\mathrm{Sp}(2l,k)$, where char(k) is odd. In Lie theory, a parabolic is obtained by fixing a subset of simple roots~\cite[Section 8.3]{ca}. Siegel maximal parabolic corresponds to the subset consisting of all but the last simple root. 
Geometrically, a parabolic subgroup is obtained as a fixed subgroup of a totally isotropic flag~\cite[Proposition 12.13]{mt}. The Siegel maximal parabolic is the fixed subgroup of  following isotropic flag (with the basis in Section~\ref{basics}): $$\{0\}\subset \{e_1,\ldots,e_l\}\subset V.$$
Thus $P$ is of the form
         $\begin{pmatrix} \alpha&0&Y\\ E&A&B\\F&0&D\end{pmatrix}$ in  $\mathrm{O}(2l+1,k)$ 
         and $\begin{pmatrix} A&B\\0&D\end{pmatrix}$ in
         $\mathrm{Sp}(2l,k)$ and $\mathrm{O}(2l,k)$.

The problem is to get the double coset decomposition $P\backslash G/P$. That is, we want to write $G=\underset{\omega\in\widehat{W}}{\bigsqcup} P\omega P $ as disjoint union, where $\widehat{W}$ is
a finite subset of $G$. Equivalently, given $g\in G$ we need an
algorithm to determine the unique $\omega\in \widehat{W}$ such that $g\in
P\omega P$. 
If $G$ is connected with Weyl group $W$ and suppose $W_P$ is the Weyl group corresponding to $P$ then~\cite[Proposition 2.8.1]{ca2}  
$$ P\backslash G/P \longleftrightarrow W_P\backslash W/W_P.$$
We need a slight variation of this as the orthogonal group is not connected. We define $\widehat{W}$ as follows:

\[\widehat{W}=\left\{\omega_0=I, \omega_i=w_{1}\cdots w_{i}\mid 1\leq i\leq l\right\},\]
where $w_i$ were defined earlier for each class of groups.

\subsection{Proof of~\ref{thmbruhatdecomposition}}\label{proofbruhat}
 In this proof, we proceed with a similar but slightly different Gaussian
 elimination algorithm. Recall that
 $g=\begin{pmatrix}A&B\\C&D\end{pmatrix}$ whenever $g$ belongs to
 $\text{Sp}(2l,k)$ or $\text{O}(2l,k)$ or
 $g=\begin{pmatrix}\alpha&X&Y\\ E&A&B\\F&C&D\end{pmatrix}$ whenever
 $g$ belongs to $\text{O}(2l+1,k)$. In our algorithm, we made $A$ into a diagonal
 matrix. Instead of that, we can use elementary operations ER1 and EC1 to
 make $C$ into a diagonal matrix and then do the row interchange to
 make $A$ into a diagonal matrix and $C$ a zero matrix. If we do that,
 we note that elementary matrices $x_{i,j}(t)$ and $x_{i,-j}(t)$ are in
 $P$. The proof is just keeping track of elements of
 $P$ in this Gaussian elimination algorithm. The Step 1 in the
 algorithm says that there are elements $p_1, p_2\in P$
 such that $p_1gp_2 = \begin{pmatrix} A_1&B_1\\C_1&D_1\end{pmatrix}$, 
 where $C_1$ is a diagonal matrix with $m$ non-zero entries. Clearly, 
 $m=0$ if and only if $g\in P$. In that case, $g$ is in the double
 coset $P\omega_0P=P$. Now suppose $m\geq 1$. Then in Step 2 we
 multiply by $x_{i,-j}(t)$ to make the first $m$ rows of $A_1$ zero, i.e., there is a $p_3\in P$ such that $p_3p_1gp_2 = \begin{pmatrix} \tilde A_1 & \tilde B_1\\C_1&D_1\end{pmatrix}$, where first $m$ rows of $\tilde A_1$ are zero. After this, we interchange rows $i$ with $-i$ for $1\leq i\leq m$ which makes $C_1$ zero, i.e., multiplying by $\omega_m$ we get  $\omega_m p_3p_1gp_2=\begin{pmatrix} A_2&B_2\\0&D_2\end{pmatrix}\in P$. Thus $g\in P\omega_m P$.
 
 For $\mathrm{O}(2l+1,k)$, we note that the elementary matrices $x_{i,j}(t), x_{i,-j}(t)$
 and $x_{i,0}(t)$ are in $P$. Rest of the proof is similar to the earlier case
 and follows by carefully keeping track of elementary matrices used in
 our algorithm in Section~\ref{wordproblem}.
 
 \medskip
 \textbf{Acknowledgement:} We are indebted to the anonymous referees for their careful reading. Their comments and suggestions improved this paper. This work was supported by a SERB research grant.

\bibliographystyle{amsplain}
\bibliography{biblio_gauss_alg}
\nocite{shinde}
\end{document}